\documentclass[11pt]{amsart}
\usepackage{amsmath,amssymb,latexsym,soul,cite,amsthm,color,enumitem,graphicx,tikz, mathtools}
\usepackage[expansion=false]{microtype}

\usepackage{float}

\usepackage[colorlinks=true, urlcolor=blue, citecolor=red, linkcolor=blue]{hyperref}
\usepackage[english]{babel}
\usepackage[left=2.5cm,right=2.5cm,top=2.7cm,bottom=2.7cm]{geometry}

\numberwithin{equation}{section}

\newtheorem{theorem}{Theorem}[section]
\theoremstyle{plain}
\newtheorem{lemma}[theorem]{Lemma}
\theoremstyle{plain}
\newtheorem{proposition}[theorem]{Proposition}
\theoremstyle{plain}
\newtheorem{corollary}[theorem]{Corollary}

\theoremstyle{definition}
\newtheorem{remark}[theorem]{Remark}

\newcommand{\N}{{\mathbb N}}

\newcommand{\R}{{\mathbb R}}

\newcommand{\beq}{\begin{equation}}
\newcommand{\eeq}{\end{equation}}

\renewcommand{\ge}{\geqslant}

\usepackage{amsmath}
\def\Xint#1{\mathchoice
   {\XXint\displaystyle\textstyle{#1}}%
   {\XXint\textstyle\scriptstyle{#1}}%
   {\XXint\scriptstyle\scriptscriptstyle{#1}}%
   {\XXint\scriptscriptstyle\scriptscriptstyle{#1}}%
   \!\int}
\def\XXint#1#2#3{{\setbox0=\hbox{$#1{#2#3}{\int}$ }
    \vcenter{\hbox{$#2#3$ }}\kern-.6\wd0}}

\def\dashint{\Xint-} % This defines \dashint

\def\Q{\mathcal{Q}}
\def\K{\mathcal{K}}

 \def\B{\mathcal{B}}

\def\P{\mathcal{P}}

\def\S{\mathcal{S}}

\title[Anisotropic Trudinger equation]{The Weak Harnack inequality and the rigidity  \\ $\quad$\\ of the Anisotropic Trudinger's equation}
\author[S. Ciani \& M. Galeotti \& I.I. Skrypnik]{S. Ciani \& M. Galeotti \& I.I. Skrypnik}

\address{Department of Mathematics of the University of Bologna, Piazza Porta San Donato, 5, 40126 Bologna, Italy}
\email{simone.ciani3@unibo.it {\&} mattia.galeotti4@unibo.it}
\address{Institute of Applied Mathematics and Mechanics, National Academy of Sciences of Ukraine, Gen. Batiouk Str. 19, 84116 Sloviansk, Ukraine}
\email{ihor.skrypnik@gmail.com}
\begin{document}
%%%%%%%%%%%%%%%%%%%%%%%%%%%%%%%%%%%%%%%%%%%%%%%%%%%
%%%%%%%%%%%%%%%%%%%%%%%%%%%%%%%%%%%%%%%%%%%%
\begin{abstract} 
\noindent We prove a local weak Harnack inequality for nonnegative weak super-solutions to the anisotropic Trudinger equation. As an application, we show a proof of the Harnack inequality that bypasses any sort of Krylov-Safonov covering argument. We further develop an analysis of global sub-potential lower bounds which ultimately leads to the identification of a Barenblatt-type profile.  

%%%%%%%%%%%%%%%%%%%%%%%%%%%%%%%%%%%%%%%%%%%%%%%%%%%

%%%%%%%%%%%%%%%%%%%%%%%%%%%%%%%%%%%%%%%%%%%%%%%%%%%
\end{abstract}

\maketitle

	\begin{center}
		\begin{minipage}{9cm}
			\small
			\tableofcontents
		\end{minipage}
	\end{center}

%%%%%%%%%%%%%%%%%%%%%%%%%%%%%%%%%%%%%%%%%%%%%%%%%%%%%%%%%%%%%%%%%%%%%%%

\section{Introduction}
\subsection{Heuristics and description of the problem} In the early twentieth century, L.S. Leibenson investigated the flow of oil and gas through porous media, motivated by problems arising in the oil fields surrounding the city of Baku. One of his most significant contributions concerns the mathematical description of turbulent gas filtration in porous media. In his pioneering work \cite{Liebenson}, he refers to the  parabolic equation \begin{equation} \label{leibenson} \partial_t u^m =c \Delta_p u , \qquad \text{for} \quad (x,y,z,t) \in \R^3 \times(0,T),\end{equation}
which seems, to the extent of our knowledge, to be the first one to appear in the scientific literature. A detailed historical account of the origins of this equation, together with its physical motivation and subsequent mathematical development, can be found in \cite{Origin}. In the particular model considered by Leibenson, the constitutive parameters satisfy \[m+1=p=3/2\,,\] 
classifying in the modern literature the equation as a $(p-1)$-homogeneous doubly nonlinear diffusion equation. This special kind of parabolic equations (with general exponent $m=p-1>0$) has been addressed also as Trudinger's equation, since the method of J. Moser (see \cite{Moser}), refined by Neil Trudinger to the nonlinear case in \cite{Tru}, could be applied successfully in order to obtain a Harnack estimate. The problem has been then studied on the wider perspective of doubly nonlinear equations in \cite{Kalashnikov}, and since then has received a huge impetus in this research field. We refrain here from describing the subject as a whole, that would drive us out of the railway of our presentation. In the model described by Leibenson, he prescribed the same diffusion along the three coordinates $(x,y,z)$, while for some materials the diffusion should have different power-laws for different directions, that tangle in a competitive effort as a sum of monotone operators (see \cite{Ant-Sh}, \cite{Lions}). The description of the regularity properties of solutions of such type of operators is precisely the scope of the present work. Here we continue the study of \cite{CHSS}, addressing anisotropic operators that, similarly to \eqref{leibenson}, enjoy a special homogeneity that balances the energetic contributions in time and space as
\begin{equation} \label{prototype} \partial_t u-\sum\limits_{i=1}^N \partial_i\Big(u^{2-p_i}|\partial_i u|^{p_i-2} \partial_i u\Big)=0,\quad u\geqslant 0.\end{equation} More precisely, let $\Omega\subseteq \mathbb{R}^{N}$ be open, $T>0$ and consider numbers $\{p_i\}_{i=1}^N$ such that\vskip0.3cm \noindent 
\[ 1<p_1:=\min\limits_{1\leqslant i\leqslant N}p_i\leq p_i \leq \max\limits_{1\leqslant i\leqslant N}p_i=:p_N\, .\]
Let also $p$ be the harmonic average of the $\{p_i\}_{i=1}^N$, and let us assume
\[ p:=\bigg( \frac{1}{N}\sum\limits_{i=1}^N\frac{1}{p_i}\bigg)^{-1}<N\,.\]
\noindent Observe that when $p_i=X$ for all $i=1,\dots, N$ then the harmonic average coincides with $X$, i.e. $p=X$; but also that, equation \eqref{prototype} is not exactly fitting with \eqref{leibenson}, since it is {\it orthotropic}, i.e.
\[\partial_t u-\sum\limits_{i=1}^N \partial_i\Big(u^{2-p}|\partial_i u|^{p-2} \partial_i u\Big)=0\,\]
shows a modulus of ellipticity that blows up when $p<2$ anytime that a single partial derivative of $u$ vanishes. Nevertheless, the contribute of $u$ vanishing itself may lead to re-balance the anisotropic competition effect. An essential point in this framework would be to understand how to remove the parabolic zero set of the solution, see for instance \cite{Rado1}, \cite{Rado2}, \cite{Rado3}; or parabolic anisotropic rigidity properties as for instance Liouville-type properties (see \cite{CG}) or growth-constraints as in   \cite{Friedman}, \cite{Glagoleva1}, \cite{Glagoleva2}. The crucial point of these being, nevertheless, a weak Harnack-type inequality. We refer to {\it the data} as the set $\{N, p_i\}$.

\subsection{Main Results} Our first result is a parabolic weak Harnack estimate, valid for nonnegative super-solutions to \eqref{prototype}. For the geometric notation we refer to Section \ref{Notation}.
\begin{theorem} \label{thm-WH}
Let us assume $Q_{\rho}= \K_{4\rho}(y) \times (s,\, s+(4\rho)^p] \subset \Omega_T$ for some $\rho>0$. If $u\ge 0$ is a local weak super-solution of \eqref{prototype} in $Q_{\rho}$, then there exist numbers $\delta_o, q \in (0,1)$ and $\gamma>1$ depending only on the data such that 

\begin{equation}\label{weak-harnack}
\bigg( \dashint_{\K_{\rho}(y)} u^q (x,s)\, dx \bigg)^{1/q} \leq \gamma \inf_{\K_{2\rho}(y)} u(\cdot, \, t),    
\end{equation} for all times
\begin{equation}
    \label{10}
    s+ \frac{\delta_o}{2} \rho^p < t < s+ \delta_o \rho^p\,.
\end{equation}
\end{theorem} \noindent As a first consequence, we show a proof of the full Harnack estimate that bypasses the Krylov-Safonov trick, see Section \ref{sec-Harnack}. Finally, we use the Harnack estimate to prove the following global sub-potential lower bounds.
\begin{theorem} \label{thm-sub}
Let $u$ be a positive local weak solution to \eqref{prototype} in $\R^{N+1}$. There exists a constant $\gamma>0$ depending only on the data, such that for all $(x,t),(y,s) \in \R^{N+1}$ with $t>s$ we have
    \begin{equation}\label{splb-clean}
        u(x,t) \ge u(y,s) \, \, \gamma^{-1} \mathcal{B} (x,t,y,s), 
    \end{equation}
    where $\B$ is the function
    \[ \B(x,t,y,s)= \exp \bigg( -\sum_i \bigg(\frac{|x_i-y_i|^{p_i}}{
    t-s} \bigg)^{\frac{1}{p_i-1}}\bigg)\,.\]
\end{theorem}
\noindent As a simple application, Theorem \ref{thm-sub} suggests a natural guess for the construction of a special self-similar solution. Indeed, yet Trudinger's equation (see \cite{S}) \begin{equation} \label{Trudy} (|u|^{p-2}u)_t=\text{div}(|Du|^{p-2}Du),\ 1<p<\infty\,, \end{equation} 
\noindent has a corresponding "fundamental solution", which is nowadays called Barenblatt's solution since its discovery in \cite{Barenblatt}, which is 
\begin{equation} \label{funda-iso}
\mathcal{B}_p(x,t)=\frac{1}{t^{N/p}}\exp \left(-\xi \left(\frac{|x|}{t^{1/p}}\right)^\frac{p}{p-1}\right),\end{equation}
\noindent with $\xi =(p-1)^2 p^{-\frac{p}{p-1}}.$ Here we show that, similarly to the solutions to \eqref{Trudy}, positive local weak solutions to \eqref{prototype} are bounded from below by a function that reflects the anisotropic nature of the equation, and that indeed can be manipulated to be a subsolution, see \eqref{eq_Bsol}. We believe that this strategy can be useful in order to construct subsolutions to more general anisotropic operators (see for instance  \cite{JEE}, \cite{impact}, \cite{CMV}).

\subsection{Novelty and Significance}
In the isotropic framework ($p_i\equiv p$), a proof of the weak Harnack inequality can be found in \cite{KK}, or adapted from \cite{Liao2} (see also \cite{GiaVes} and \cite{Tru}). On the other hand, the anisotropic context needs a careful analysis of the anisotropic energy (see \eqref{3}). For equations with non-standard growth, Theorem \ref{thm-WH} seems to be new. For what concerns the point-wise estimates of Theorem \ref{thm-sub}, since for doubly nonlinear equations the Harnack inequality and the H\"older continuity are two independent issues, it is not known if local weak solutions to \eqref{prototype} are continuous. In the isotropic case some answers have been given in \cite{KK2}, \cite{K1}, \cite{K2} and in \cite{VerLia} for sign-changing solutions, since in the case where $u>\eta >0$ the equation behaves, roughly speaking, as the $p$-Laplacean equation - whose solutions are nowadays known to be continuous. In the case of equation \eqref{prototype}, only a partial answer is given in \cite{CHSS} through the technique of \cite{JEE} that restricts qualitatively the sparseness of the exponents $p_i$s. The crucial point is that it is not known if bounded solutions to the anisotropic $p$-Laplacian 
\[\partial_t u - \sum_i \partial_i (|\partial_i u|^{p_i-2}\partial_i u)=0\]
are continuous. Hence, the trick does not apply. For this reason and other technical purposes, we show that local sub-solutions have an upper-semicontinuous representative, see Section \ref{sec-semi}. Finally, the kind of sub-potential lower bounds we establish here (Theorem \ref{thm-sub}) are derived through a Harnack chain procedure reminiscent of \cite{DBGV-mono}. The difference lies on the fact that the procedure has to be performed on each separate direction at each step.

\subsection*{Structure of the paper}
In Section \ref{Notation} we collect the notation used throughout the paper, while Section \ref{Preliminaries} is devoted to set up the functional framework and to recall some measure-theoretical properties of the solutions which are necessary to our study, such as energy estimates, measure propagation in time, and expansion of positivity. In Section \ref{WHSection} we prove the weak Harnack inequality and its use for a new proof of the Harnack inequality. Finally, Section \ref{Section-Sub-Potential} is devoted to point-wise sub potential bounds and the construction of a model solution, while the brief Section \ref{generalization} explains the extent of our method.

\section{Notation} \label{Notation} \subsection{Indexes} We assume that the numbers $\{p_1,p_2, \dots, p_N\}$ are fixed. In order to ease notation we further assume, without loss of generality, that 
\[1<p_1 \leq p_2 \leq \dots \leq p_N< \infty\] and we set
\[\frac{1}{p}:=\frac{1}{N}\sum\limits_{i=1}^N\frac{1}{p_i}\,.\]
When index $i$ is concerned, we will always reduce the set of indexes as in the previous formula, also for the product $\prod_{i=1}^N= \prod_i$. Given $K_1,K_2$ by the structure conditions \eqref{eq1.2}, we will refer to the set of parameters $\{N, p_1, \dots, p_N, K_1,K_2\}$ as the set of (structural) {{\it data}, and by saying that a constant $\gamma$ depends only on the data we mean that it can be quantitatively determined {\it a priori} only in terms of the above quantities. As usual, the generic constant $\gamma$ may change from line to line. 
\subsection{Geometry} For fixed $r, \eta >0$ and a point $(\bar{x},\bar{t}) \in \R^{N+1}$, we define the standard cube
\[K_r(\bar{x})= \big\{x\in \R^N:\quad |x_i-\bar{x}_i|<r, \quad i=1, \dots, N \big\} \ , \]
the anisotropic cubes
\[\K_{r}(\bar{x}):=\big\{x\in \R^N:\quad |x_i-\bar{x}_i|<r^{\frac{p}{p_i}},\quad i=1, ...,N\big\},\]
and the cylinders
\[Q_{r, \eta}(\bar{x}, \bar{t}):=\K_{r}(\bar{x})\times (\bar{t}-\eta, \bar{t}).\]
Finally, for fixed $k,r,\eta>0$, we define the intrinsic anisotropic cubes 
\[\K^k_r(\bar{x}):=\big\{x\in \R^N:\quad  |x_i-\bar{x}_i|<r^{\frac{p}{p_i}} k^{-\frac{p_N-p_i}{p_i}},\quad i=1, ..., N\big\} \ ,\]
and cylinders
\[Q^k_{r, \eta}(\bar{x}, \bar{t}):=\K^k_r(\bar{x})\times (\bar{t}-\eta, \bar{t}).\]

\section{Functional Setting and Auxiliary Results}\label{Preliminaries}

\subsection{Functional Anisotropic Setting and Parabolic Embeddings}
\noindent Equation \eqref{prototype} has to be understood in an appropriate variational setting. For this aim, we define the anisotropic spaces of locally integrable functions as
$$W^{1,\textbf{p}}(\Omega)=\big\{u\in W^{1,1}(\Omega): D_i u\in L^{p_i}(\Omega),\quad i=1, ...,N\big\},$$
$$L^{\textbf{p}}(0, T; W^{1,\textbf{p}}(\Omega))=\big\{u\in L^1(0, T; W^{1,1}(\Omega)): D_i u\in L^{p_i}(0, T; L^{p_i}(\Omega)),\quad i=1, ...,N\big\},$$
and the respective spaces of functions with zero boundary data
$$W^{1,\textbf{p}}_0(\Omega)=W^{1,1}_0(\Omega)\cap W^{1,\textbf{p}}(\Omega),$$
$$L^{\textbf{p}}(0, T; W^{1,\textbf{p}}_0(\Omega))=L^{1}(0, T; W^{1,1}_0(\Omega))\cap L^{\textbf{p}}(0, T; W^{1,\textbf{p}}(\Omega)).$$

\subsection{Definition of weak solution} \label{def-sol} Here below we give the notion of local weak solution to
equation \eqref{prototype}. For the motivation lying behind this definition and an alternative definition that takes into account an approximated time-derivative, we refer to the Appendix of \cite{CHSS} and \cite{Vestberg}. A measurable function $u : \Omega_T \rightarrow [0, \infty)$ such that
\begin{equation} \label{functional}
\begin{cases} 
u\in C_{loc}(0, T; L^{\frac{p_N}{p_N-1}}_{loc}(\Omega)),\quad u^{\frac{1}{p_N-1}}\in L^{\textbf{p}}_{loc}(0, T; W^{1,\textbf{p}}_{loc}(\Omega)),\\\quad \\
D_i (u^{\frac{1}{p_i-1}} ) \in L^{p_i}_{loc}(0, T; L^{p_i}_{loc}(\Omega)),\,\qquad \text{for all} \quad i=1, ...,N,
\end{cases}
\end{equation}
is a  non-negative, local, weak sub(super)-solution to \eqref{prototype}, if for every compact set $E\subset \Omega$ and every sub-interval $[t_1, t_2]\subset (0, T]$ it satisfies the inequality 
\begin{equation}\label{eq1.3}
\int\limits_E u\,\varphi\,dx \Big|^{t_2}_{t_1}+\int\limits^{t_2}_{t_1}\int\limits_E \Big[u\, \varphi_t+\sum_i A_i(x, t, u, D u)\,D_i \varphi \Big]\,dx dt \leqslant (\geqslant) 0,
\end{equation}
for all non-negative testing functions
$$\varphi\in W^{1,p_N}_{loc}(0, T; L^{p_N}(E))\cap L^{\textbf{p}}_{loc}(0,T; W^{1,\textbf{p}}_0(E)).$$
As usual, we say that $u$ is a non-negative, local weak solution to equation \eqref{prototype} if $u$ is both a non-negative, local weak
sub- and super-solution.

\subsection{Local Energy Estimates}

\begin{lemma}\label{lem2.6}
Let $u$ be a non-negative, local weak sub(super)-solution to  \eqref{prototype}. Then there exists a constant $\gamma >0$, depending only on the data, such that, for every cylinder $Q_{r, \eta}(y, \tau)\subset \Omega_T$, every $k\in \mathbb{R}_+$, and every piecewise smooth, cutoff function $\zeta$ vanishing on $\partial K_{r}(y)$ and such that $0\leqslant \zeta \leqslant 1$, there holds
\begin{multline}\label{energy-estimates}
\sup\limits_{\tau-\eta\leqslant t \leqslant \tau}\int\limits_{K_{r}(y)}g_{\pm}(u^{\frac{1}{p_N-1}}, k^{\frac{1}{p_N-1}})\,\zeta^{p_N}\,dx\, +\\+\gamma^{-1}  \sum\limits_{i=1}^N\iint\limits_{Q_{r,\eta}(y, \tau)}u^{\frac{p_N-p_i}{p_N-1}}|D_i (u^{\frac{1}{p_N-1}}-k^{\frac{1}{p_N-1}})_{\pm}|^{p_i}\zeta^{p_N}\,dx\,dt\\\leqslant \int\limits_{K_{r}(y)\times\{\tau-\eta\}}g_{\pm}(u^{\frac{1}{p_N-1}}, k^{\frac{1}{p_N-1}})\,\,\zeta^{p_N}\,dx+\gamma\iint\limits_{Q_{r,\eta}(y, \tau)} g_{\pm}(u^{\frac{1}{p_N-1}}, k^{\frac{1}{p_N-1}})\,|\zeta_t|\,dx dt\, +\\+\gamma \sum\limits_{i=1}^N\iint\limits_{Q_{r,\eta}(y, \tau)}u^{\frac{p_N-p_i}{p_N-1}}(u^{\frac{1}{p_N-1}}-k^{\frac{1}{p_N-1}})_{\pm}^{p_i}|D_i \zeta|^{p_i}\,dx\,dt\, .
\end{multline} where
$$g_{\pm}(u^m, k^m):=\pm\,\frac{1}{m}\int\limits^{u^m}_{k^m} s^{\frac{1}{m}-1} (s-k^m)_{\pm}\,ds , \qquad k,m>0 \ .$$
\end{lemma}

\noindent The doubly-nonlinear nature of \eqref{prototype} is embodied in the terms $g_{\pm}$, that can be estimated through the following technical tool. 
\begin{lemma}\label{lem2.5}
There exists a constant $\gamma >0$, depending only on $m$, such that
\begin{equation*}
\frac{1}{\gamma}\,(k^m+u^m)^{\frac{1}{m}-1}(u^m-k^m)^2_{\pm}\leqslant g_{\pm}(u^m, k^m)\leqslant \gamma\, (k^m+u^m)^{\frac{1}{m}-1}(u^m-k^m)^2_{\pm}\,.
\end{equation*}
\end{lemma} \noindent Next Lemma states that, as a typical parabolic feature, the measure information can be transported forward in time. We traced the detailed proof of \cite{CHSS} in order to extrapolate the dependence on the data of the assumption.

\begin{lemma}\label{lem4.1}
Let $u\ge 0$ be a local weak super-solution to \eqref{prototype} in $Q=K_8(0)\times (0,1)$ and assume that 
\begin{equation}\label{hp-measure-forward}
   \left|\left[u(\cdot, 0)\ge k\right]\cap K_1(0)\right| \ge \alpha_o |K_1|\,, 
\end{equation}  for some $\alpha_o\in (0,1)$ and $k \in \R$.
 Then there exist $\varepsilon_1$, $\delta_1 \in (0, 1)$ depending only on the data and on $\alpha_o$, such that
\begin{equation}\label{eq4.5}
|K_1(0)\cap [u(\cdot, \tau)\geqslant \varepsilon_1\,k]|\geqslant \frac{\alpha_o}{4}|K_1(0)|,
\end{equation}
for all times
\begin{equation*}
0 \leqslant \tau\leqslant \delta_1\,.
\end{equation*} The dependence of $\varepsilon_1, \delta_1$ from $\alpha_o$ is 
\[\delta_1 \approx \frac{\alpha_o^{p_N+1} (1-\alpha_o)}{\gamma}\quad \text{and} \quad \varepsilon_1 \approx \frac{\alpha_o}{\gamma}\,,\]
for a constant $\gamma>0$ depending only on the data $\{N, p_i\}$.
\end{lemma}
\noindent Finally, as a main Tool of the Trade for the Weak Harnack inequality, we recall the following {\it expansion of positivity} Lemma, obtained through a refinement of the De Giorgi level set method. The latter, being based on the additivity properties of the measure, does not provide a power-like dependence on the data.

\begin{lemma}\label{EP}
Let $u$ be a non-negative, local weak super-solution to \eqref{prototype} and let us assume that for some $r$, $k>0$ and $\alpha_p \in (0,1)$ there holds
\begin{equation}\label{EP-HP}
\left|
\left[u(\cdot, s)\geqslant k\right]
\cap\K_r(y)\right|\geqslant \alpha_p |\K_r(y)|.
\end{equation}
Then, there exist numbers $\eta_p$, $\delta_p \in (0, 1)$ depending only on {  the data} and $\alpha_p$ such that
\begin{equation}\label{EP-TH}
u(x, t)\geqslant \eta_p\,k,
\end{equation}
for all $(x,t)$ respecting
\begin{equation}\label{EP-times}
x\in \K_{2r}(y),\qquad
s+\frac{\delta_p}{2}r^p\leqslant t \leqslant s+\delta_p r^p,
\end{equation}
provided that 
$$\K_{8r}(y)\times(s, s+r^p)\subset \Omega_T.$$
\end{lemma}
\noindent Finally, the following classic iteration Lemma will be very useful.

\begin{lemma}\label{lem_iteration} Let $\{Y_n\}$ be a sequence of bounded numbers such that, for constants $\S,b>1$ and $\varepsilon \in (0,1)$
\begin{equation}\label{sqn}
    Y_n \leq \varepsilon Y_{n+1}+ \S b^n \ , 
\end{equation} 
\noindent with $\varepsilon b<1$. Then, by a simple iteration, there exists $\gamma >0$ such that 
\[Y_0 \leq \gamma\,  \S.\]
\end{lemma}

\section{Weak Harnack inequality} \label{WHSection}
\noindent 
\begin{lemma}[Local Clustering] \label{lem-local-clustering} For $y \in \R^N$ and $r>0$, let $u \in W^{1,1}(K_r(y))$ be a function satisfying
\begin{equation}\label{local-clustering-hp}
\int_{K_r(y)} |\nabla (u-k)_{-}| \, dx \leq \gamma_c k r^{N-1}, \quad \text{and} \quad |[u \ge k] \cap K_r(y)| \ge \beta |K_r|\,,    
\end{equation} for some numbers $\beta \in (0,1)$, $k \in \R$ and $\gamma_c >0$. Then, for any $\lambda,\nu \in (0,1)$ there exists a point $x_c \in K_r(y)$ and a constant $\delta_c=\delta_c(N) \in (0,1)$ such that 
\begin{equation}\label{th-local-clustering}
\left|[u \leq \lambda k]     \cap K_{\rho_c}(x_c)| \leq \nu |K_{\rho_c} (x_c) \right| \,, \quad  \text{for} \quad \rho_c= r \bigg( \frac{\delta_c \beta^2(1-\lambda)\nu}{\gamma_c}\bigg)\,.
\end{equation}
\end{lemma} 
\noindent The next proposition lies at the heart of the parabolic Weak Harnack inequality; it has to be formulated in an appropriate anisotropic geometry.

\begin{proposition}[Refined Expansion of Positivity] \label{REP} Let $u\ge 0$ be a local weak solution of \eqref{prototype} and let us assume that for some $r,k>0$, $\alpha_o \in (0,1)$ the following measure information is at stake,
\begin{equation}\label{hp-ref-exp-pos}
|[u(\cdot, s)>k] \cap \K_r(y)| \ge \alpha_o |\K_r|\,.
\end{equation} Then, there exist constants $\eta_o, \delta_o \in (0,1)$, $d>1$,  depending only on the data $\{N, p_i\}$, such that
\begin{equation}\label{th-ref-exp-pos}
    u(\cdot, t) \ge \eta_o \, \alpha_o^d\, k \qquad \text{a.e. in} \quad \K_{2r}(y)\,,
\end{equation} for all times
\[s+ \frac{\delta_o}{2} r^p \leq t \leq s+ {\delta_o} r^p\,.\]
\end{proposition}

\begin{proof} Let us assume without loss of generality that $(y,s)=(0,0)$. By Lemma \ref{lem4.1}, there exist $\delta_1(\alpha_o), \varepsilon_1(\alpha_o) \in (0,1)$ such that 
\begin{equation}\label{1}
| [u(\cdot, t) \ge \varepsilon_1(\alpha_o) k] \cap \K_r(0)| \ge \frac{\alpha_o}{4} |\K_r|\,,
\end{equation} for all $0 \leq t \leq \delta_1(\alpha_o)$. Now we set
\[\Q'= \K_{2r}(0) \times (0, \delta_1 r^p], \qquad \text{and} \qquad \Q=K_r(0)\times \bigg(\frac{\delta_1}{2} r^p, \, r^p\bigg]\,.\]
We apply Lemma \ref{lem2.6} to $(u-k)_{-}$ over the pair of cylinders $\Q', \Q$, and manipulate the parabolic terms with 
Lemma \ref{lem2.5} and with a function $\zeta \in C_o(\Q')$ 
respecting the following conditions,
\[0\leq \zeta \leq 1, \quad \zeta|_{\Q}\equiv 1,\quad |D_i\zeta|\leq \gamma/ r^{\frac{p}{p_i}},\quad |\partial_t \zeta |\leq \gamma / \delta_1.\]
We obtain the following energy bound,
\begin{equation}\label{eq_eb2}
    \begin{aligned}
\sum_i  \iint_{\Q} &u^{\frac{p_N-p_i}{p_N-1}} |D_i (u^{\frac{1}{p_N-1}}-k_1^{\frac{1}{p_N-1}})_{-}|^{p_i} \, dxdt \\
&\leq \iint_{\Q'} \bigg{\{}(k_1^{\frac{1}{p_N-1}}+u^{\frac{1}{p_N-1}}) (u^{\frac{1}{p_N-1}}-k_1^{\frac{1}{p_N-1}})_{-}^2 + \sum_i u^{\frac{p_N-p_i}{p_N-1}}(u^{\frac{1}{p_N-1}}-k_1^{\frac{1}{p_N-1}})_{-}^{p_i} | D_i \zeta|^{p_i} \bigg{\}}\, dxdt   \\
& \qquad \leq \frac{\gamma k_1^{\frac{p_N}{p_N-1}}}{\delta_1(\alpha_o)} |\K_{2r}| = \frac{\gamma (\varepsilon_1 \,k)^{\frac{p_N}{p_N-1}}}{\alpha_o^{p_N+1}} |\K_{2r}|\,, 
    \end{aligned}
\end{equation} 
where $k_1= \varepsilon_1\, k$. Now we take into consideration the following change of variables
\begin{equation}\label{eq_cov}\tau:=\frac{t}{\delta_1 r^p}\,,\quad    z_i:=\frac{x_i}{r^{p/p_i}}\,,\quad \text{and }\ v(z,\tau):=
    \frac{u(x,t)}{\varepsilon_1 k}.\,\end{equation}
%\begin{equation*}
%    \begin{cases}
%        t \rightarrow t / (\delta_1 r^p)=: \tau,\\
%        x_i \rightarrow {x_i}/ r^{\frac{p}{p_i}} = z_i\,,
%    \end{cases} \text{and} \qquad v(z,\tau)=\frac{u(t,s)}{\varepsilon_1 k}\,.
%\end{equation*} 
The measure information 
\eqref{1}, transforms into
\begin{equation} \label{2}
|[v(\cdot, \tau) >1] \cap K_1| \ge \frac{\alpha_o}{4} |K_1|\,, \qquad \forall\quad  {1/2}<\tau<1\,,
\end{equation} and the energy bound \eqref{eq_eb2}, turns in the following inequality, which is independent of $k$,
\begin{equation}
    \label{3}
\sum_i \int_{1/2}^1 \int_{K_1}  v^{\frac{p_N-p_i}{p_N-1}}\,\left|D_{z_i} ( v^{\frac{1}{p_N-1}}-1)_{-}\right|^{p_i} \, dzd\tau \leq \frac{\gamma}{\alpha_o^{p_N+1}}\,\,.
\end{equation} \noindent Our aim now is to bound \eqref{3} from below by the norm $\|\cdot \|_{W^{1,1}}$ of the function $v^{\alpha}$ where 
\[\alpha= 1+\frac{1}{p_N-1}.\] 
%is a power greater than one.
Indeed, by H\"older inequality and by setting $\hat{Q}= K_1 \times (1/2,\, 1)$, we get
\begin{equation}\label{eq:boundvalpha}
    \begin{aligned}
        \iint_{\hat{Q}} &|\nabla (v^{\alpha}-1)_{-}| \, dzd\tau \\
        & \leq  \sum_i  \iint_{\hat{Q} \cap [v\leq 1]} v^{\alpha-1} |D_i v|\, dzd\tau \\
        & \leq  \sum_i  \iint_{\hat{Q} \cap [v\leq 1]}  \bigg( v^{\frac{2-p_i}{p_N-1}} |D_i v|^{p_i} \, dz d\tau \bigg)^{\frac{1}{p_i}} \bigg( \iint_{\hat{Q} \cap [v\leq 1]} v^{(\frac{2-p_i}{p_N-1})(\frac{1}{p_i-1})+ \frac{(\alpha-1)p_i}{p_i-1}} \,\, dzd\tau\bigg)^{\frac{p_i-1}{p_i}}\\
        & \leq \sum_i \bigg( \iint_{\hat{Q} \cap [v\leq 1]} \left|D_i(v^{\frac{1}{p_N-1}}-1)_{-}\right|^{p_i} \, dzd\tau \bigg)^{\frac{1}{p_i}} \, \left|\hat{Q} \cap [v\leq 1]\right|^{1-\frac{1}{p_i}} \\
        & %\qquad \qquad \qquad \qquad  \quad 
        \leq {\gamma}\, {\alpha_o^{-(\frac{p_N+1}{p_1})} }\,.
    \end{aligned}
\end{equation}

\begin{lemma}
There exists a time $\bar{\tau} \in (1/2,\, 1)$ such that
\begin{equation}\label{4}
\int_{K_1\times \{ \bar{\tau}\}} |\nabla (v^{\alpha}-1)_{-} |\, dx 
\leq \gamma 2^N {\alpha_o^{-\frac{p_N+1}{p_1}}}  \qquad \text{and} \qquad |[v(\cdot, \bar{\tau}) \ge 1] \cap K_1| \ge \frac{\alpha_o}{4} |K_1|\,.
\end{equation} 
\end{lemma}
\begin{proof} 
%We can suppose $\alpha_o \leq 1/2$, and 
The second inequality was
already proved at \eqref{2}.
To prove the first inequality, we define the set
\begin{align*}
T_1&:=\bigg{\{} \tau \in (1/2, \, 1):  \int_{K_1 \times \{\tau\}} |\nabla (v^{\alpha}-1)_{-} | \, dx
> \gamma2^N\alpha_o^{-\frac{p_N+1}{p_1}}  \bigg{\}},
%T_2&:=\bigg{\{}\tau \in (1/2, \, 1): |[v(\cdot,t)>1] \cap K_1| \ge \frac{\alpha_o}{8} \bigg{\}}.
\end{align*} 
From the definition of $T_1$ and \eqref{eq:boundvalpha}, we have 
\[
{2^N\gamma}  |T_1| {\alpha_o^{-\frac{p_N+1}{p_1}}}
\leq  \iint_{\hat{Q}}  |\nabla (v^{\alpha}-1)_{-} | (\cdot,\tau ) \, dz\, d\tau 
\leq {\gamma}{\alpha_o^{-\frac{p_N+1}{p_1}}}, \quad \Rightarrow \quad |T_1|<2^{-N}.
\]
%On the other hand, by the definition of $T_2$,
%\begin{equation*}
%    \begin{aligned}
%    \alpha_o/4 < |[v>1]
%    \cap \hat{Q}|
%    &= \int_{1/2}^1  |[v(\cdot,\tau)>1] \cap K_1| \, d \tau \\
%    &= \int_{T_2} |[v(\cdot,\tau)>1] \cap K_1| \, d\tau + \int_{T_2^c} |[v(\cdot,\tau)>1] \cap K_1|\, d\tau \\
%    &\leq |T_2|+ \alpha_o/8,\\
%    &\Rightarrow |T_2|>\alpha_o/8\,.    \end{aligned}
%\end{equation*} 
This proves the claim since we have
\[
|T_1^c|>1-\frac{1}{2^N}>0.
\]
%This proves the claim, since $|T_1^c|+|T_2| >1/2$.
\end{proof}

\noindent The sliced information \eqref{4} now is perfect to apply the local clustering Lemma \ref{lem-local-clustering} to the function $v^{\alpha}$ with $\lambda=1/2$ and $\nu=1/2$,
obtaining that there exists a point $x_c \in K_1$ where the measure information clusters:
\begin{equation} \label{5}
    \left|
    \left[v^{\alpha}(\cdot, \bar{\tau}) >\dfrac12\right] \cap K_{\sigma}(x_c)\right|\, >\frac{|K_{\sigma}|}{2} \,, \qquad \sigma = \overline\gamma\, \alpha_o^{2+\frac{p_N+1}{p_1}}\,,
\end{equation} for a constant $\overline \gamma>0$ that depends only on the data $\{N,p_i\}$.
We can 
now consider the inverse coordinate change
with respect to \eqref{eq_cov}, and get
a point $y_c \in \K_r$ and a time 
\[
\bar{t} \in \left(\delta_1\frac{r^p}{2}, \delta_1 r^p\right)
\]
such that the following measure information
is verified,
\begin{equation}
    \label{6}
    \left|
    \left[u(\cdot, \bar{t}) > 2^{-\frac{1}{\alpha}}\,\varepsilon_1 k\right] \cap \K_{\sigma r}(y_c)\right|
    \ge \frac{|\K_{\sigma r}|}{2}\,,
\end{equation}
with $\varepsilon_1 \approx \alpha_o/\gamma\,$ as
stated in Lemma \ref{lem4.1}.
\noindent From here, we apply iteratively the expansion of positivity Lemma~\ref{EP} with $\alpha_p = 1/2$, to obtain, two constants $\eta_p, \delta_p \in (0,1)$ depending only on the data $\{N, p_i\}$, such that

\begin{equation} \label{7}
u(\cdot, t) \ge 2^{-\frac{1}{\alpha}}\,\varepsilon_1 k\, \eta_p^n \qquad \text{a.e. in }\K_{2^n \sigma r} (y_c),
\end{equation}
for $n \in \mathbb{N}$ and for all times 
\[
\bar t+\frac{\delta_p}{2} R^p < t<\bar t+\delta_p R^p
\]
with 
\begin{equation}\label{8}
R^p=\sum_{j=0}^{n-1} (2^j \sigma r)^p =\sigma^pr^p \left( \frac{2^{pn}-1}{2^p-1}  \right)\,.
\end{equation}

\noindent Fix $L>3$ 
a constant that will be made explicit later,
and let $n \in \mathbb{N}$ be sufficiently 
{big that $2^n \sigma \geq L$, \textit{i.e.}, $n \geq \log_2(L/\sigma)$.}
Recall that $y_c\in \K_r(0)$,
therefore we achieve the inclusion
\[
\K_{2r}(0)\subset\K_{Lr}(y_c)\subset \K_{2^n \sigma r}(y_c),
\]
and since\[\varepsilon_1 k \eta_p^n= \varepsilon_1 k 2^{n\,\log_2(\eta_p)}= \varepsilon_1 k \bigg(\frac{L}{\sigma} \bigg)^{\log_2(\eta_p)}=\frac{\alpha_o}{\gamma} k \bigg(\frac{\alpha_o^{2+(p_N+1)/p_1}}{\gamma L} \bigg)^{\log_2(\eta_p^{-1})}\]
we obtain from \eqref{7} the lower bound
\begin{equation}\label{9}
    u(\cdot, t) \ge \eta_f \, \alpha_o^d\, k\,, 
\end{equation} with 
\[ \eta_f= \frac{2^{-\frac{1}{\alpha}}}{\gamma\overline\gamma^{\log_2(\eta_p)}} L^{\log_2(\eta_p)}\qquad \text{and} \qquad d= 1+\bigg[2+ \bigg(\frac{p_N+1}{p_1} \bigg)\log_2(\eta_p^{-1})\bigg]\,.\]
\noindent Moreover, since 
$\bar{t} \in \left( \delta_1 \frac{r^p}{2}, \, \delta_1 r^p\right)$, and because of \eqref{8}, we get
\[
\frac{\delta_1}{2} (r^p+R^p) < t< 
%\delta_p r^p  \bigg( \frac{L^p-2^p\sigma^p}{2^p-1} \bigg)
\delta_1 (r^p+R^p)\,,
\]
therefore the proposition is proved after retransforming to the original coordinate system, with 
\[
\delta_o= \delta_1\left(1+\sigma^p\left(
\frac{2^{np}-1}{2^p-1}\right)\right)\,,
\]
where $\delta_1$ is provided in Lemma \ref{lem4.1}.   
\end{proof}

\begin{proof}[Proof of Theorem \ref{thm-WH}]
\noindent The proof is nowadays standard but we briefly sketch it here for the sake of completeness, and to trace the dependence on $\delta_o$. Let $(y,s)=(0,0)$ and let us define
\[I:=  \inf_{\K_{2\rho}\times (\delta_o \rho^p/2,\, \delta_o \rho^p)} u   \, .
\]
\noindent The layer-cake formula implies
\begin{equation} \label{11}
    \begin{aligned}
        \int_{\K_{\rho}} u^q(x,0)\, dx &= q \int _0^{\infty} \left|[u(\cdot, 0)>M] \cap \K_{\rho}\right| \, M^{q-1}\, dM\\
        &\leq q \int_I^\infty
        \left|[u(\cdot, 0)>M] \cap \K_{\rho}\right| \, M^{q-1}\, dM + I^q \, |\K_{\rho}|\,.
    \end{aligned}
\end{equation} By Proposition \ref{REP} applied  with 
\[
\alpha_o= \frac{|[u(\cdot, 0) >M] \cap \K_{\rho}|}{|\K_{\rho}|},
\]
there exists a number $\eta_o \in (0,1)$ such that 
\[ I \ge \eta_o M \bigg(\frac{|[u(\cdot, 0) >M]\cap \K_{\rho}|}{|\K_{\rho}|} \bigg)^d\,\]
therefore, if we choose $q<1/d$, then the integral in \eqref{11} converges and gives the desired estimate
\begin{align*}
    \int_{\K_{\rho}} u^q(x, 0)\, dx 
    &\leq q \int_I^\infty \bigg( \frac{I}{\eta_o M} \bigg)^{1/d} |\K_{\rho}| M^{q-1}\, dM + I^q |\K_{\rho}| \\
    &\leq \frac{q}{\eta_0^{1/d}}\int_I^\infty I^{1/d}\, |\K_\rho|M^{q-1-1/d}\, dM+|\K_\rho| I^q\\
    &\leq |\K_\rho|\left(\bar C+1\right)I^q\,.
\end{align*}\end{proof}

\subsection{Weak Harnack estimates imply the Harnack inequality} A consequence of the weak Harnack inequality is the Harnack inequality, when for subsolutions it is possible to obtain a refined $L^{\infty}$-bound. Hence, first we prove the aforementioned refined $L^{\infty}$ estimate, and next we show that sub-solutions always have a well-defined upper-semicontinuous representative. The two considerations together with the weak Harnack inequality will provide the full Harnack inequality.
\subsubsection{Refined $L^{\infty}$ bound}
\begin{theorem} \label{refined} Let $(x_o,t_o) \in \Omega_T$ and $\theta, \rho>0$ be such that $(x_o,t_o)+Q_{4 \rho, \theta (4 \rho)^p} \subset \Omega_T$. If $u \ge 0$ is a sub-solution to equation \eqref{prototype}, then 
\begin{equation}\label{refined-linfty}
    \gamma^{-1} \sup_{(x_o,t_o)+Q_{\rho,\theta \rho^p}} u \leq \theta^{\frac{N}{p}(\frac{1-p_N}{p_N})} \bigg( \sup_{t_o-\theta \rho^p<\tau< t_o} \dashint_{x_o+\K_{\rho}} u^q(x,\tau) \, dx \bigg)^{1/q} 
\end{equation}
    
\end{theorem}

\begin{proof} Let $(x_o,t_o)=(0,0)$ and let us refine the estimate of \cite[Lemma 3.7]{CHSS}, which states that there exist two constants $\gamma, \alpha>0$ depending only on the data $\{N,p_i\}$ such that for any $\sigma \in (0,1)$ we can estimate
\begin{equation}\label{eq:sup}
     \sup_{Q_{\rho_0, \, \theta (\rho_0)^p}} u  \leq \gamma \bigg[\bigg( \frac{1}{\theta\, |\rho_1-\rho_0|^p} \bigg)^{\frac{N+p}{p}} \iint_{Q_{\rho_1, \, \theta \rho_1^p}} u^{\frac{p_N}{p_N-1}} \, dxdt \bigg]^{\frac{p_N-1}{p_N}}\,.
\end{equation} 
Let's define, for every $n\in \N$,
\[
Q_{n}:=Q_{(2-2^{-n}) \rho, \, \theta ((2-2^{-n}) \rho)^p}, \qquad 
M_n:=\sup_{Q_n}u,\qquad
\S:= \sup_{-\theta \rho^p < \tau< 0}\, \dashint_{\K_{\rho}} u^q(x,\tau)\, dx  \,\,.
\] 
Moreover we observe that
$2-2^{-n-1}-(2-2^{-n})=2^{-n-1}$,
and we set
\[
\mu:={\left(\frac{1}{2}\right)}^{(N+p)\frac{p_N-1}{p_N}}.
\]
%\[
%\frac{2-2^{-n}}{2-2^{-n-1}}=1-\mu_n\ \Rightarrow \mu_n=\frac{2^{-n-1}}{1-2^{-n-1}}.
%\]
Therefore, by \eqref{eq:sup} and Young's inequality,
we get
\begin{equation}\label{12}
    \begin{aligned}
M_n=\sup_{Q_n} u 
&\leq \frac{\gamma}{\mu^n} \bigg[ \frac{1}{\theta^{\frac{N+p}{p}}}\frac{1}{\rho^{N+p}} \iint_{Q_{n+1}} u^{\frac{p_N}{p_N-1}} \, dxdt \bigg]^{\frac{p_N-1}{p_N}}\\
& \leq \frac{\gamma}{\mu^n} \bigg[ \theta^{-\frac{N+p}{p}}\bigg( \sup_{Q_{n+1}} u \bigg)^{\frac{p_N}{p_N-1}-q}  \bigg( \theta \sup_{- \theta \rho^p<\tau<0} \dashint_{\K_{\rho}} u^{q}(x,\tau) \, dx \bigg)\bigg]^{\frac{p_N-1}{p_N}}\\
&= \frac{\gamma}{\mu^n} \theta^{-\frac{N}{p}(\frac{p_N-1}{p_N})}  M_{n+1}^{1-\frac{q(p_N-1)}{p_N}} \S^{q(\frac{p_N-1}{p_N})}\\
&\leq \varepsilon M_{n+1}+\frac{\gamma C(\varepsilon)}{\overline \mu^n}\S
    \end{aligned}
\end{equation} for eventually  a different $\gamma >0$.
Here
\[
\overline \mu=\mu^\frac{p_N}{q(p_N-1)}=\left(\frac{1}{2}\right)^{\frac{N+p}{q}}<1,
\]
therefore we can iterate Lemma \ref{lem_iteration} with $Y_n= M_n,\ b= \overline\mu^{-1}$, and this finishes the job. 
\end{proof}
\subsubsection{Semicontinuity} \label{sec-semi}
In this subsection we prove that a local weak subsolution $u \ge 0$ to equation \eqref{prototype} has an upper-semicontinuous representative. Observe that, if $-u$ would have been a super-solution to \eqref{prototype}, then our result would have been a consequence of Lemma (De Giorgi-Type) of \cite{CHSS} and \cite{Liao}. Nevertheless, the structure of equation \eqref{prototype} does not allow this trick, so we need to adapt the reasoning of the latter to our specific case. The main tool needed is a De Giorgi type Lemma in whole cylinders, recalling that nonnegative local weak sub-solutions to \eqref{prototype} are locally bounded. This Lemma can be found in \cite{JEE}, Lemma 2.5, by setting $m=1/(p_N-1)$.

\begin{lemma}[De Giorgi-type Lemma] \label{DG-lemma} Let $a, \xi \in (0,1)$ and $u \ge 0$ a local weak sub-solution to \eqref{prototype} in $\Omega_T$. Let $(x_o,t_o) \in \Omega_T$ and $\rho >0$ be a radius such that 
\[ Q_{4\rho}(x_o,t_o) = \K_{4\rho}(x_o) \times (t_o-(4\rho)^p, \, t_o + (4\rho)^p) \subseteq \Omega_T\,\]
and let us define numbers
\[ \mu^+= \sup_{Q_{4\rho}(x_o,t_o)} \, u \, , \qquad \text{and} \qquad  \omega= \mu^+- \inf_{Q_{4\rho}(x_o,t_o)} \, u \, .\]
\noindent Then, there exists a number $\nu \in (0,1)$ depending only on the data and $a,\xi, \mu^+, \omega$, but not on $\rho$, such that if the measure information
\begin{equation}\label{measure-DG+}
\left|Q_{\rho}(x_o,t_o) \cap\bigg[ u^{\frac{1}{p_N-1}} \ge [\mu^+]^{\frac{1}{p_N-1}}- [\xi \omega ]^{\frac{1}{p_N-1}}\bigg] \right| \leq \nu |Q_{\rho}|,    
\end{equation} is valid, then 
\begin{equation}\label{TH-DG}
    \sup_{Q_{\rho/2}} \, u^{\frac{1}{p_N-1}}\, \, \leq [\mu^+]^{\frac{1}{p_N-1}}- [a \xi \omega]^{\frac{1}{p_N-1}} \, \, .
\end{equation}
\end{lemma}

\begin{proposition} \label{semicontinuity} Local weak sub-solution $u \ge 0$ to equation \eqref{prototype} is upper-semicontinuous. \end{proposition}

\begin{proof} The upper semicontinuous regularisation of $u$ in $\Omega_T$ can be defined as
\[u^*(x,t)= \text{ess} \, \limsup_{(y,s) \rightarrow (x,t)}\,u(y,s)\, = \lim_{\rho \downarrow 0} \left( {\text{ess}\sup}_{Q_{\rho}(x,t)}\, u  \right), \qquad \text{for all} \quad (x,t) \in \Omega_T\,.\]
 Since $u\ge 0$ is a local weak sub-solution to \eqref{prototype} then $u$ is locally bounded \cite{JDE} and $u^*: \Omega_T \rightarrow \R$ is well-defined at every point. Since $u \in L_{loc}^1(\Omega_T)$, we introduce the set of Lebesgue points of $u$ in $\Omega_T$, 
 \[\mathcal{L}(u,\Omega_T) := \bigg{\{} (x,t) \in \Omega_T\,:\, \lim_{\rho \downarrow 0} \dashint \dashint_{Q_{\rho}(x,t)} |u(y,s)-u(x,t)|\, dyds =0\,\bigg{\}}\,.\]
Observe
that $Q_{\rho}$ is the ball in $\R^{N+1}$ of radius $\rho^p$ with respect to the distance
\[ d\left((x,t),\, (y,s)\right)= \max \bigg{\{} |t-s|,\, |x_i-y_i|^{p_i}:\ i=1, \dots, N\bigg{\}},\]
and that $\Omega_T$ with the Lebesgue measure and this metric is a metric measure space. Hence by  
\cite{BjornBjorn} we have $|\mathcal{L}(u,\Omega_T)| =|\Omega_T|$. We want to show that $u=u^*$ on a subset of full measure of $\Omega_T$, in particular we choose to prove
the double inequality on $\mathcal L(u,\Omega_T)$. It is obvious that $u^* \ge u$ in $\mathcal{L}(u,\Omega_T)$ since
\[u^*(x,t)= \lim_{\rho \downarrow 0} \bigg( \sup_{Q_{\rho}(x,t)}\, u \bigg) \ge \lim_{\rho \downarrow 0} \dashint\dashint_{Q_{\rho}(x,t)} u(y,s)\, dyds = u(x,t)\,.\]
\noindent Let us suppose, 
by contradiction,
that the converse inequality does not hold at $(x_o,t_o) \in \Omega_T$, i.e.
\[u(x_o,t_o) < u^*(x_o, t_o)\,.\]
Let $R>0$ be such that 
\[ Q_{R}(x_o,t_o)= \K_{R}(x_o) \times (t_o-R^p,\, t_o+R^p) \subset \Omega_T,\]
and for 
\[\mu^+= \sup_{Q_R(x_o,t_o)} \, u\,,\qquad \omega= \mu^+- \inf_{Q_R(x_o,t_o)}\, u\,,\]
let $\xi \in (0,1)$, be such that 
\begin{equation}
    u(x_o,t_o)^{\frac{1}{p_N-1}}< [\mu^+]^{\frac{1}{p_N-1}}-[\xi \omega]^{\frac{1}{p_N-1}} < u^*(x_o,t_o)^{\frac{1}{p_N-1}}\leq [\mu^+]^{\frac{1}{p_N-1}}\,.
\end{equation} 
\noindent We can now choose $a \in (0,1)$ such that
\begin{equation}\label{eq:assumption}
    [\mu^+]^{\frac{1}{p_N-1}}-[a\xi \omega]^{\frac{1}{p_N-1}} < u^*(x_o,t_o)^{\frac{1}{p_N-1}}\,.
\end{equation}
If $\nu \in (0,1)$ is that of Lemma \ref{DG-lemma}, we claim that for some small radius $\rho \in (0,R/4)$,
the condition
\eqref{measure-DG+}
must hold.
Indeed,
if this is not the case, then
\begin{equation*}
    \begin{aligned}
        \iint_{Q_{\rho}(x_o,t_o)} &|u(x_o,t_o)- u(y,s)| \,dyds \\
        &\ge \iint_{\left\{u^{\frac{1}{p_N-1}} \ge [\mu^+]^{\frac{1}{p_N-1}}-[\xi \omega]^{\frac{1}{p_N-1}}\right\}} \bigg[  \bigg([\mu^+]^{\frac{1}{p_N-1}}-[\xi \omega]^{\frac{1}{p_N-1}}\bigg)^{p_N-1}- u(x_o,t_o)   \bigg]\, dyds \\
        &\ge \nu |Q_{\rho}| \, \bigg[  \bigg([\mu^+]^{\frac{1}{p_N-1}}-[\xi \omega]^{\frac{1}{p_N-1}}\bigg)^{p_N-1}- u(x_o,t_o)   \bigg] >0\,,
    \end{aligned}
\end{equation*} and by letting $\rho$ vanish we contradict the assumption $(x_o,t_o) \in \mathcal{L}(u,\Omega_T)$. We consider $\rho_0 \in (0,R/4)$ such that \eqref{measure-DG+} is valid. By applying Lemma \ref{DG-lemma} and 
considering the assumption \eqref{eq:assumption},
we obtain
 the estimate
\[
\sup_{Q_{\rho_0/2}}u^{\frac{1}{p_N-1}} \leq [\mu^+]^{\frac{1}{p_N-1}}-[a\xi\omega]^{\frac{1}{p_N-1}} < u^*(x_o,t_o)^{\frac{1}{p_N-1}},\]
which
contradicts the very definition of $u^*(x_o,t_o)$.
\end{proof}

 \noindent Finally we can prove the Harnack inequality, without using any Krylov-Safonov covering argument. 

\subsubsection{Proof of the Harnack's inequality}\label{sec-Harnack} Let us fix $(x_o,t_o) \in \Omega_T$ and $\rho>0$ such that 
 
 \[ \K_{4\rho, \theta (4\rho)^p}(x_o) \times (t_o- \theta (4\rho)^p, \, t_o) \subset \Omega_T\,.\]
 By Theorem \ref{refined} we have that there exists a constant $\gamma>0$ such that 

 \begin{equation*}
\gamma^{-1} \sup_{Q_S} u \leq  \bigg( \sup_{t_o-\theta \rho^p<t<t_o} \, \, \dashint_{\K_{\rho}(x_o)} u^q(x,t)\, dx \bigg)^{1/q}\,,
 \end{equation*} for 
 \[Q_S= (x_o,t_o)+ Q_{\rho, \theta \rho^p}\,.\]
 \noindent Now, since $u:\Omega_T \rightarrow [0, \infty)$ has an upper-semicontinuous representative $u^*$, there exists a time $\tau_{\S} \in [t_o-\theta \rho^p, \, t_o]$ such that the supremum in time of the function $t \rightarrow \|u (\cdot, t)\|_{q, \K_{\rho}}$ is attained,
\begin{equation}\label{12.2}
\gamma^{-1} \, \, \sup_{Q_S} u  
\leq  \, \, 
\bigg( \dashint_{\K_{\rho}(x_o)} u^q(x,\tau_{\S})\, dx \bigg)^{1/q}\,.
\end{equation}
Now we apply the weak Harnack inequality \eqref{weak-harnack} to the right-hand side of \eqref{12.2}, %\eqref{eq:sup},
to obtain 

 \begin{equation}\label{H}
 \sup_{Q_S} u \leq \gamma\, \,  \inf_{Q_I} u\, ,      
 \end{equation} with 
 \[Q_I= (x_o,\tau_S)+ \K_{2\rho} \times \left(\delta_o \rho^p /2, \, \delta_o \rho^p\right)\,.\]

\medskip

\noindent By choosing $\theta=\frac{\delta_o}{4}$, we get that necessarily (See Figure \ref{figAA})
\[
\left[t_o+\frac{\delta_o}{2}\rho^p,\, t_o+\frac{3\delta_o}{4}\rho^p\right]\subset \left[\tau_S+\frac{\delta_o}{2}\rho^p,\, \tau_S+\delta_o\rho^p\right]\,,
\]
hence, 
\begin{equation}
    \label{Harnack2}
\sup_{(x_o,t_o)+Q_{\rho,\frac{\delta_o}{4}\rho^p}}u\leq 
\gamma  \inf_{(x_o,t_o)+\K_{\rho}\times(\frac{\delta_o}{2}\rho^p,\frac{3\delta_o}{4}\rho^p)}u \,.
\end{equation}

\begin{figure}    
    \centering

\begin{tikzpicture}[>=stealth, line width=0.9pt, line join=round]

  % ---- Rettangolo esterno (cilindro grande) ----
  \draw (-3,-3.2) rectangle (3,2.8);

  % ---- Q_J : rettangolo superiore (interno, sopra l'asse x) ----
  \draw[fill=blue!20] (-1.7,1.1) rectangle (1.7,2.2);
  \node at (1.0,1.65) {$Q_I$};

  % ---- Q_S : rettangolo inferiore (interno, lato superiore = asse x) ----
  \draw[fill=blue!20]  (-1.7,-2.8) rectangle (1.7,0);
  \node at (1.0,-1.4) {$Q_S$};

  % ---- Assi (continuano a sinistra e in basso) ----
  \draw[->] (0,-3.8) -- (0,3.4)  node[left=2pt] {$t$};   % asse t
  \draw[->] (-5.3,0) -- (5.3,0)  node[above right] {$x$}; % asse x

  % ---- Formula: livello del lato inferiore del rettangolo grande ----
  \node at (-0.75,-3.5){\tiny $t_o-\theta(4\rho)^p$};

\end{tikzpicture}

    \caption{Representation of Harnack's inequality in the form \eqref{Harnack2}}\label{figAA}
\end{figure}
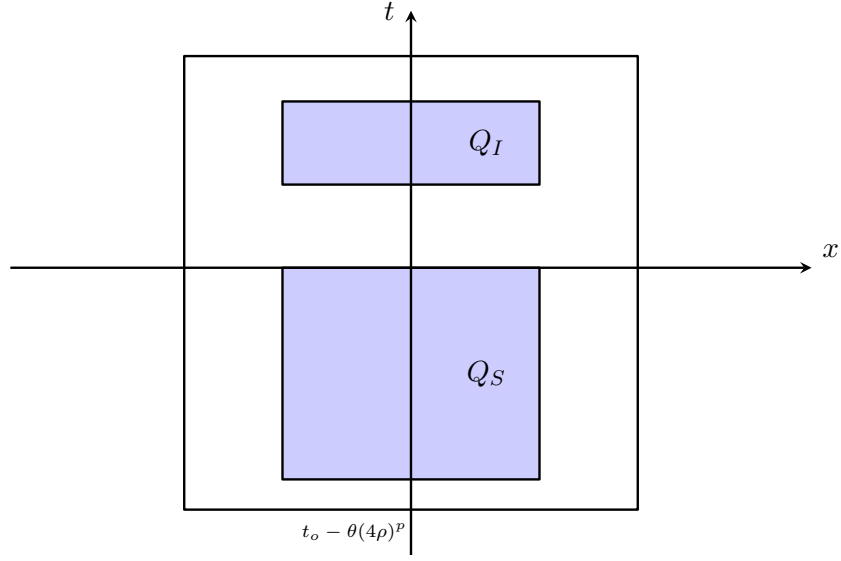

\begin{remark}
    Instead of selecting an upper semicontinuous representative $u^*$ of $u$ we may, in principle, just choose a smaller $q$ in Theorem \ref{thm-WH}, i.e. $q<p_N/(p_N-1)$. In this way, by the very definition of solution, one has $u\in C(0,T; L^q_{loc}(\Omega))$. The importance of selecting $u^*$ allows for the chaining procedure to  be done for every $0<q<1/d$. Moreover, we use $u^*$ to make sense of \eqref{HI} and the point-wise estimates of Section \ref{Section-Sub-Potential}.
\end{remark}
\section{Sub-Potential Lower Bounds and Time-Decay} \label{Section-Sub-Potential}
\noindent  If we set $\theta = 3\delta_o/4$ we obtain from \eqref{Harnack2} the following version of the Harnack inequality (c.f. \cite{CHSS}),
\begin{equation}\label{HI}
 u(x_0, t_0)\leqslant C\, \inf\limits_{\K_\rho(x_0)} u(\cdot,\,  t_0+\theta\,\rho^{p}),
\end{equation} for nonnegative solutions to \eqref{prototype}, provided that the following inclusion is satisfied
\[
\K_{8\rho}(x_0) \times (t_0- \theta (8\rho)^{p}, t_0+ \theta(8\rho)^{p}) \subset \Omega_T \ . \]

% \begin{theorem}\label{thm-Harnack}
% Let $u$ be a non-negative, local weak solution to \eqref{prototype} in $\Omega_T$. There exist positive constants $C$, $\theta$,  depending only on the data such that, for all cylinders \[
% \K_{8\rho}(x_0) \times (t_0- \theta (8\rho)^{p}, t_0+ \theta(8\rho)^{p}) \subset \Omega_T \ , \] there holds
% \begin{equation}\label{HI}
% \frac{1}{C}\sup\limits_{\K_\rho(x_0)}u(\cdot,\,  t_0-\theta \,\rho^{p})\leqslant u(x_0, t_0)\leqslant C\, \inf\limits_{\K_\rho(x_0)} u(\cdot,\,  t_0+\theta\,\rho^{p}),
% \end{equation} 
% where
% \[
% \K_{\rho}(x_0):=\prod_{i=1}^N \bigg\{|x_i-{x_0}_i|<\rho^{\frac{p}{p_i}}\bigg\}\,.\]
% \end{theorem}
% \noindent  

\begin{lemma} Let $u$ be a positive local weak solution of \eqref{prototype} in $\Omega_T= \Omega \times (0,T]$ with $\Omega$ open bounded set of $\R^N$. Let $(x,t)$, $(y,s)$ be points in $\Omega_T$ satisfying $s>t$ and let $\rho$ be the biggest number such that \[\K_{8\rho}(y) \times (s, s+\theta (8\rho)^p)\subset \Omega\times (0,T]\,.\]  Then we have
\begin{equation}\label{sblb-monster}
      u(y,s) \ge u(x,t) \, \, \gamma^{-1} \exp\bigg(- \bigg[ \frac{(s-t)}{\theta \rho^p}+\sum_i \bigg(\frac{|x_i-y_i|}{(8\rho)^{1/p_i}}  \bigg)  +  \bigg(\frac{|x_i-y_i|^{p_i}}{(s-t)} \bigg)^{\frac{1}{p_i-1}} \bigg]\bigg)\,.
\end{equation}  
\end{lemma}
\begin{proof}
Given a generic point $(x,t)$, the right-hand side of the Harnack estimate \eqref{HI} is valid in the forward parabolids
\[
\P(x,t)^+= \bigg{\{} (z,\tau) \in 
\R^{N+1}:\quad |x_i-z_i|^{p_i} \theta \leq 
(\tau-t), \quad \tau \ge t \,,\, \forall i=1,\dots N \bigg{\}}\,.
\]
We assume $(x,t)=(0,0)$ and join the origin to the point $(y,s)$ by a straight line (see Figure \ref{figA}) and set the following sequence of points for $j=1,\dots n$, with $n \in \N$ to be determined,
\[
    x_n:=y,\qquad
    t_n:=s,\qquad
    x_j:=\frac{j}{n}\,x_n,\qquad
    t_j:=\frac{j}{n}\,t_n.
\]
See Figure \ref{figA} for an illustration. We will determine conditions on number $n \in \mathbb{N}$ for each segment joining $(x_j,t_j)$ to $(x_{j+1}, t_{j+1})$ to be contained in an appropriate paraboloid $\P(x_j,t_j)$. Hence, the conditions will be
\begin{equation}
    \begin{cases}
        |x_{j+1,i}-x_{j,i}|=\frac{|y_i|}{n} < (8\rho)^{\frac{1}{p_i}},\\ 
        \quad \\
        t_{j+1}-t_j = \frac{s}{n} < \theta \rho^p,\\ \quad \\
        \frac{|y_i|^{p_i}}{n^{p_i}} = |x_{j+1,i}-x_{j,i}|^{p_i} \leq \frac{t_{j+1}-t_j}{\theta}= \frac{s}{n \theta},
    \end{cases} \quad \Rightarrow \qquad 
    \begin{cases}
        n > |y_i|/(8\rho)^{1/p_i},\\ \quad \\
        n>s/(\theta \rho^p),\\ \quad \\
        n \ge \bigg( \frac{\theta |y_i|^{p_i}}{s}  \bigg)^{\frac{1}{p_i-1}}, \quad \forall i=1, \dots,N\,.
    \end{cases}
\end{equation} \noindent When the conditions on $n$ are respected, the Harnack inequality \eqref{HI} implies $u(x_j,t_j) \leq C u (x_{j+1}, t_{j+1})$ and 
\begin{equation} \label{chain} \frac{u(0,0)}{u(y,s)}=\frac{u(0,0)}{u(x_1,t_1)}\, \frac{u(x_1,t_1)}{u(x_2,t_2)}\,\dots \,  \frac{u(x_j, t_j)}{u(x_{j+1}, t_{j+1})}\, \dots \, \frac{u(x_{n-1}, t_{n-1})}{u(y,s)} \leq C^n\,.\end{equation} Hence, by choosing 
\begin{equation} \label{n} n= 1+ \frac{s}{\theta \rho^p}+ \sum_i \frac{|y_i|}{(8\rho)^{\frac{1}{p_i}}}+ \bigg( \frac{\theta |y_i|^{p_i}}{s} \bigg)^{\frac{1}{p_i-1}},\end{equation}
the conditions above are satisfied and, by taking  the logarithm 
\[\ln\left( \frac{u(0,0)}{u(y,s)}\right) \leq   \ln (C) \, n\,, \]
we finally obtain the claim. 
\noindent The use of the Harnack inequality in $\Omega_T$ is justified by the choice of $\rho$ while, in case $u$ solves the equation in the whole $\R^{N+1}$, we can choose $\rho= \infty$ in our estimate and recover \eqref{splb-clean}.
 
\begin{figure}
\centering
\begin{tikzpicture}[scale=0.4]
  % Axes
  \draw[->] (-8,0) -- (10,0) node[below right] {$x$};

  % Fill region ABOVE the paraboloid curve
  
\draw[->] (0,-0.5) -- (0,17) node[above] {$t$};
  % Paraboloid curve: t = 7|x|^3
  \draw[thick, blue] (-6,15) parabola bend (0,0) (6,15);
    
  % Label the curve
  \node[blue, right] at (6, 17) {$\P(x_o,t_o)$};

  % External point (x, t) outside the paraboloid
  \filldraw[black] (8, 9) circle (2pt) node[above right] {$(x, t)$};
 \node[black] at (0, -1) {$(x_o, t_o)$};

 \draw[thick, black] (0,0) -- (8,9);

 \node[black , right] at (8, 8) {$\ell$};

\filldraw[violet] (2.7, 3.1) circle (2pt);
\node[violet] at (4, 2.7) {$(x_1, t_1)$};

\draw[thick, violet] (-2.5,14.7) parabola bend (2.7, 3.1) (9,14.7);

 \node[violet , right] at (10, 13) {$\P(x_1,t_1)$};

\end{tikzpicture}

\caption{Scheme of the proof of the sub-potential lower bound \eqref{splb-clean}, by reaching a point $(x,t)$ by repeated application of \eqref{HI} in the paraboloids $\P(x_j,t_j)$.}\label{figA}
\end{figure}
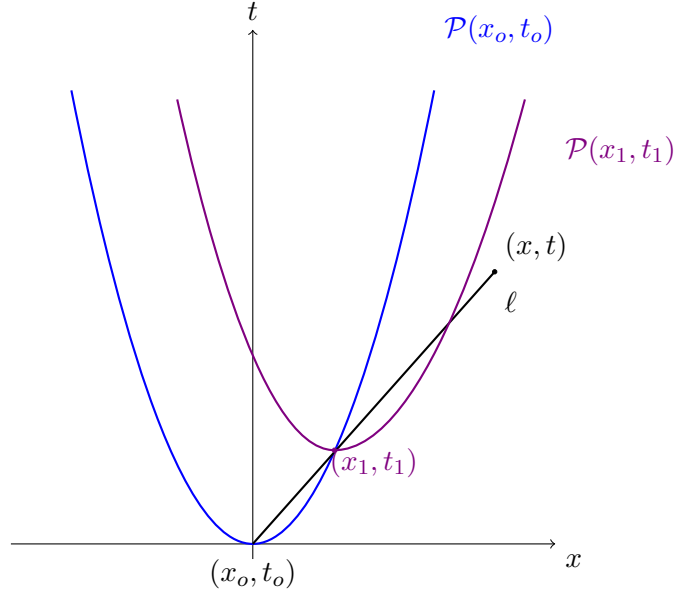
\end{proof}

\noindent The above sub-potential lower-bound can be improved to a proper decay with respect to time.

\begin{corollary}
    Consider $u(x,t)$ a solution of \eqref{prototype}
    and
    $r>0$ a positive real.
    Then, there exists a number $A$, such that the following is respected.
    For every $x_0,x$ with $x\in K_r(x_0)\subset K_{8r}(x_0)\subset\Omega$ and every~$t\geq t_0$,
    \begin{equation}
        u(x,t)\geq \gamma^{-1}\left(\frac{t}{t_0}\right)^A u(x_0,t_0)\B(x,t,x_0,t_0).
    \end{equation}
\end{corollary}
\begin{proof}
    We set 
    $r_*:=\max\{r:\ K_{8r}(x_0)\subset \Omega\}$, and
    $\sigma:=\frac{\theta r^p}{t_0}$.
    Then, consider
    \[
    k=\left\lfloor\log_{1+\sigma}\left(\frac{t}{t_0}\right)\right\rfloor, 
    \]
    in particular $k$ is the only integer number such that $(1+\sigma)^kt_0\leq t<(1+\sigma)^{k+1}t_0$. By \eqref{splb-clean}, 
    \[
    u(x,t)\geq \gamma^{-1}u(x_0,\tau)\B(x,t,x_0,\tau),
    \qquad
    \text{where }\tau:=(1+\sigma)^{k-1}t_0.
    \]
    Here we apply the Harnack 
    inequality \eqref{HI} iteratively
    to the right hand side. We set
    $t_j:=(1+\sigma)^jt_0$ for every $j=0,1,\dots,k-1$, 
    thus giving $\tau=t_k$.
    Then, we have
    $t_{j+1}=t_j+\sigma t_j\leq t_j+\theta r^p$, we respect the hypothesis
    of the inequality and we can state,
    \[
    u(x_0,\tau)\geq\left(\frac{1}{C}\right)^{k-1} u(x_0,t_0).
    \]
    We conclude by 
    setting
    $
    A:=-\log_{1+\sigma}(C)$.
    Indeed, we observe that
    \[
    \B(x,t,x_0,\tau)\geq \eta\,\B(x,t,x_0,t_0)\geq0
    \]
    for an opportune constant $\eta>0$,
    because $t-\tau \geq t\,\frac{\sigma}{1+\sigma}\geq (t-t_0)\frac{\sigma}{1+\sigma}$, and 
    moreover $k\leq \log_{1+\sigma}(t/t_0)$
    by definition, thus implying
    \[
    \left(\frac{1}{C}\right)^{k-1}\leq C\,(1+\sigma)^{A\cdot\log_{1+\sigma}(t/t_0)}=C\,\left(\frac{t}{t_0}\right)^A.
    \]
\end{proof}

\subsection{Construction of a Barenblatt profile}
We observe that the function $\B$ is self similar, hence its special form can be used to guess a good candidate for the expansion of positivity by comparison (see for instance \cite{CV-Barenblatt}, \cite{CV-Barenblatt}, \cite{CMV} or the original \cite{Barenblatt-book}), if only it would be a (sub)solution to \eqref{prototype}. This suggests a simple computation leading to the fact that the function 
\begin{equation}\label{eq_Bsol}
    \B(x,t)=D t^\alpha \exp\left(-\sum_ic_i{\left(\frac{|x_i|^{p_i}}{t}\right)}^{\frac{1}{p_i-1}}\right),
\end{equation}
is a solution to \eqref{prototype} for opportune choices

% of the coefficients $c_1,\,\dots,\, c_n$, and the exponent $\alpha$. We report the simple computation here below for the sake of completeness. Let's start by evaluating
% \begin{equation}\label{eq_defg}
% \partial_i\B(x,t)=\B(x,t)\left(-\frac{p_ic_i}{p_i-1}\left(\frac{|x_i|}{t}\right)^{\frac{1}{p_i-1}}\frac{x_i}{|x_i|}\right)=:
% \B(x,t)g_i(x,t),
% \end{equation}
% therefore we get
% \[
% \B(x,t)^{2-p_i}\,|\partial_i\B(x,t)|^{p_i-2}\,\partial_i\B(x,t)
% =-\B(x,t)\left(\frac{c_ip_i}{p_i-1}\right)^{p_i-1}\frac{x_i}{t}.
% \]
% We introduce 
% $\delta_i:=\left(\dfrac{c_ip_i}{p_i-1}\right)^{p_i-1}$ to ease notation, and also observe that
% \[
% \partial_t\B(x,t)=\B(x,t)\left(\frac{\alpha}{t}+\sum_i \frac{c_i}{p_i-1}\left(\frac{|x_i|}{t}\right)^{\frac{p_i}{p_i-1}}\right).
% \]
% Therefore $\B(x,t)$ is a solution
% of \eqref{prototype}
% if
% \[
% \frac{\alpha}{t}+\sum_i\frac{c_i}{p_i-1}\left(\frac{|x_i|}{t}\right)^{\frac{p_i}{p_i-1}}=
% -\sum_i\frac{\delta_i}{t}+\sum_i\delta_i^{\frac{p_i}{p_i-1}}\left(\frac{|x_i|}{t}\right)^{\frac{p_i}{p_i-1}}.
% \]
% Choosing for each $i=1, \dots, N$
\[
c_i\,:={\left({(p_i-1)^{p_i-1}}\cdot{p_i^{-p_i}}\right)}^{\frac{1}{p_i-1}} \qquad \text{and} \qquad 
\alpha\,\,=-\sum_i\delta_i\,.
\] 
% shows that $\B(x,t)$ is a solution to \eqref{prototype}.

\section{More General Structures} \label{generalization} \noindent The weak Harnack and Harnack estimates\eqref{Harnack2} and \eqref{HI} are a sole consequence of the energy inequalities of Lemma \ref{lem2.6}, and so are the sub-potential lower bounds \eqref{splb-clean}. Hence Theorems \ref{thm-WH} and \ref{thm-sub} apply, with {\it data} $=\{ N, p_i, K_1,K_2\}$ to parabolic operators of the kind of \begin{equation}\label{eq1.1}
u_{t}-\sum\limits_{i=1}^N D_i A_i (x, t, u, D u)=0, \quad \text{locally weakly in} \quad  \Omega_{T}:=\Omega \times (0, T).
\end{equation}
where $A_i(x,t,s,\xi):\Omega_{T}\times \mathbb{R}_+\times \mathbb{R}^{N} \rightarrow \mathbb{R}^{N}$ are Carath\'eodory functions, satisfying 
\begin{equation}\label{eq1.2}
\begin{cases}
 A_i(x, t, s, \xi)\, \xi_i  \geqslant  K_{1}\, s^{2-p_i}\,|\xi_i|^{p_i},\qquad \forall (s,\xi) \in \R_+ \times \R^N,\\
 \\
|A_i(x, t, s,  \xi)| \leqslant  K_{2}\, s^{2-p_i} |\xi_i|^{p_i-1},\qquad  \forall \, i=1, ...,N,
\end{cases}
\end{equation}
for positive constants $K_1,K_2$. Such operators do not necessarily satisfy the comparison principle.

\section*{Acknowledgements}

\small
\noindent 
The authors thank Eugenio Vecchi for suggestions on the construction of the Barenblatt profile. S.C.~acknowledges GNAMPA (INdAM) fundings and the support of the department of Mathematics of Bologna.

\end{document}